\documentclass[11pt,reqno]{amsart}
\usepackage{amssymb,amsthm}
\usepackage[nosumlimits]{mathtools}
\usepackage{enumerate}
\usepackage[margin=1in]{geometry}
\usepackage{eucal}

\theoremstyle{plain}
\newtheorem{theorem}{Theorem}[section]
\newtheorem{proposition}[theorem]{Proposition}
\newtheorem{lemma}[theorem]{Lemma}
\newtheorem{corollary}[theorem]{Corollary}

\theoremstyle{definition}
\newtheorem{definition}[theorem]{Definition}

\usepackage{hyperref}
\usepackage{xcolor}
\hypersetup{
    colorlinks,
    linkcolor={red!50!black},
    citecolor={blue!50!black},
    urlcolor={blue!80!black}
}

\newcommand{\m}{{\scriptscriptstyle-}}
\newcommand{\p}{{\scriptscriptstyle+}}
\newcommand{\pp}{{\scriptscriptstyle++}}
\newcommand{\tp}{{\scriptscriptstyle\mathsf{T}}}
\newcommand{\F}{{\scriptscriptstyle\mathsf{F}}}

\let\O\undefined
\DeclareMathOperator{\O}{O}

\DeclareMathOperator{\V}{V}
\DeclareMathOperator{\Sym}{\mathsf{S}}
\DeclareMathOperator{\tr}{tr}
\DeclareMathOperator{\Gr}{Gr}
\DeclareMathOperator{\diag}{diag}
\DeclareMathOperator{\ED}{\mathsf{ED}}
\DeclareMathOperator{\Flag}{Flag}

\DeclareMathOperator{\spn}{span}

\allowdisplaybreaks

\begin{document}
\title{Euclidean distance degree in manifold optimization}
\author[Z.~Lai]{Zehua~Lai}
\address{Department of Mathematics, University of Texas, Austin, TX 78712}
\email{zehua.lai@austin.utexas.edu}
\author[L.-H.~Lim]{Lek-Heng~Lim}
\address{Computational and Applied Mathematics Initiative, Department of Statistics,
University of Chicago, Chicago, IL 60637-1514.}
\email{lekheng@uchicago.edu}
\author[K.~Ye]{Ke Ye}
\address{KLMM, Academy of Mathematics and Systems Science, Chinese Academy of Sciences, Beijing 100190, China}
\email{keyk@amss.ac.cn}

\begin{abstract}
We determine the Euclidean distance degrees of the three most common manifolds arising in manifold optimization: flag, Grassmann, and Stiefel manifolds. For the Grassmannian, we will also determine the Euclidean distance degree of an important class of Schubert varieties that often appear in applications. Our technique goes further than furnishing the value of the Euclidean distance degree; it will also yield closed-form expressions for all stationary points of the Euclidean distance function in each instance. We will discuss the implications of these results on the tractability of manifold optimization problems.
\end{abstract}

\maketitle

\section{Euclidean distance degree}\label{sec:ed}

Let $\mathcal{M} \subseteq \mathbb{R}^{m \times n}$ be a  real affine variety, i.e., a set defined by polynomial equations $\mathcal{M} = \{ X \in \mathbb{R}^{m \times n} : f_1(X) = 0,\dots, f_k(X) = 0 \}$ with $f_1,\dots,f_k \in \mathbb{R}[x_{11}, x_{12}, \dots, x_{mn} ]$. Its \emph{complex locus} is the corresponding complex affine variety in $\mathbb{C}^{m \times n}$ defined by the same equations:
\[
\mathcal{M}^{\mathbb{C}} \coloneqq \{ X \in \mathbb{C}^{m \times n} : f_1(X) = 0,\dots, f_k(X) = 0 \}.
\]
In other words we add to $\mathcal{M}$ the complex solutions of its defining equations. 
The  \emph{Euclidean distance degree} \cite{DHOST14, DHOST16} of $\mathcal{M}$, denoted  $\ED(\mathcal{M})$,  is  the number of stationary points of the \emph{Euclidean distance function}
\begin{equation}\label{eq:ed}
\delta_A : \mathcal{M}^{\mathbb{C}} \to \mathbb{C}, \quad X \mapsto \tfrac12 \lVert X - A \rVert^2,
\end{equation}
for a generic $A \in \mathbb{R}^{m \times n}$, where
\[
\lVert X \rVert^2 \coloneqq \sum_{i = 1}^m \sum_{j = 1}^n x_{ij}^2 
\]
is the complexified Frobenius norm (squared). This notation is chosen merely to remind us that it agrees with the Frobenius norm
\[
\lVert X \rVert^2_\F = \sum_{i = 1}^m \sum_{j = 1}^n \lvert x_{ij}  \rvert^2
\]
when restricted to $\mathbb{R}^{m \times n}$ but $\lVert \, \cdot \, \rVert^2$ is a complex-valued polynomial function and $\lVert \, \cdot \, \rVert$ is not a norm on $\mathbb{C}^{m \times n}$. Herein lies an important rationale: By limiting ourselves to polynomials and by working over an algebraically closed field $\mathbb{C}$, we allow for the use of algebraic geometric techniques in the study of Euclidean distance degree. An optimization theorist would invariably prefer the number of stationary points in $\mathcal{M} \subseteq \mathbb{R}^{m \times n}$ with the actual Frobenius norm $\lVert \, \cdot \, \rVert^2_\F$.
But $\mathbb{R}$ is not algebraically closed and $\lVert \, \cdot \, \rVert^2_\F$ is not a polynomial.  Substituting $\lVert \, \cdot \, \rVert^2_\F$ with $\lVert \, \cdot \, \rVert^2$ and $\mathcal{M} $ with $\mathcal{M}^{\mathbb{C}}$ is a compromise. However, it gives one important piece of information for what an optimization theorist is interested in, namely, the Euclidean distance degree serves as an \emph{upper bound} for the number of stationary points in $\mathcal{M} $ with respect to the actual Frobenius norm, noting that $\lVert \, \cdot \, \rVert^2$ and $\lVert \, \cdot \, \rVert^2_\F$ take the same value on $\mathbb{R}^{m \times n}$. The requirement to work over $\mathbb{C}$ also explains why the notion is limited to stationary points, which are defined for any complex-valued functions, and not local or global optimizers, which are only defined for real-valued functions.

The Euclidean distance degree has become an increasingly popular tool in optimization theory \cite{optim1, optim2, optim3, optim4, optim5}. It was first proposed in \cite{DHOST14, DHOST16} to capture the computational complexity of a geometric object $\mathcal{M}$ embedded in an ambient Euclidean space, usually $\mathbb{R}^m$ or $\mathbb{R}^{m \times n}$ or a subspace.  The number of stationary points affects the computational complexity of an optimization problem as one way to find global optimizers of $f : \mathcal{M} \to \mathbb{R}$ is simply through brute force: Find all stationary points and then compare their values to identify the largest/smallest. There are two difficulties with this approach: Firstly, it depends not just on $\mathcal{M}$ but also the choice of  $f$. Secondly, it requires us to ascertain that we have found all stationary points. Euclidean distance degree may be viewed as a way around these difficulties.

To obtain a complexity measure for $\mathcal{M}$ independent of the objective function $f$, we choose the Euclidean distance function $\delta_A$ for a generic $A$  as a convenient stand-in for $f$. This is the default choice in pure mathematics, specifically Morse Theory \cite{Milnor}, and thoroughly understood. Furthermore, it is broadly applicable to different $\mathcal{M}$'s embedded differently in different ambient spaces. So by using the Euclidean distance function as a common yardstick, we may compare the tractability of optimizing over different $\mathcal{M}$'s, or the same $\mathcal{M}$ embedded in different ambient spaces, or the same $\mathcal{M}$ embedded in different ways in the same ambient space.

The Euclidean distance degree $\ED(\mathcal{M})$ then provides a certificate of exhaustiveness for this approach, i.e., we have found all stationary points of $\delta_A$. In addition, $\ED(\mathcal{M})$ serves as a measure of tractability: If there are exponentially many stationary points, then the brute force approach would be prohibitively expensive. 
 
Indeed one goal of our article is to use Euclidean distance degree to provide a different perspective on our intractability results in \cite{ZLK24}, in which we showed that many innocuous-looking manifold optimization problems are NP-hard. Whereas we relied on worst-case time complexity to quantify tractability in \cite{ZLK24}, here we will use Euclidean distance degree. As we will see, the Euclidean distance degrees of the flag, Grassmann, and Stiefel manifolds are all generally exponentially large. This informs us, in a way different from \cite{ZLK24}, that manifold optimization problems are expected to be hard.

\subsection*{Prior works and contributions}  Since its first appearance in \cite{DHOST14}, the Euclidean distance degree has been studied in numerous works --- see \cite{AH, BD1, BD2, BSW, DLO17, Lee17, MRW1, MRW2, MRW3, MT24, OSV} and the references therein. Here we study it for three of the most common $\mathcal{M}$ arising in manifold optimization \cite{EAS,AMS}: the flag, Grassmann, and Stiefel manifolds embedded in  $\mathbb{R}^{m \times n}$ as submanifolds. We will also be interested in certain important submanifolds of the Grassmann manifold called Schubert varieties.

The best-known existing results on Euclidean distance degrees apply to  projective varieties \cite{AH, MRW3,OSV}, groups \cite{BD1}, and hypersurfaces \cite{Lee17}. None of which apply to our submanifold models for flag, Grassmann, and Stiefel manifolds (see Section~\ref{sec:man}) or Schubert varieties (see Section~\ref{sec:Gr}), which are all affine varieties, have no group structure, and have high codimensions, aside from two special cases discussed in Section~\ref{sec:St}. We cannot apply the ``transfer principle'' in \cite{DLO17} as to apply it we need $\O_n(\mathbb{R}) \times \O_n(\mathbb{R})$-invariance in Sections~\ref{sec:Fl} and \ref{sec:Gr} but all models therein lack this property; and we need $\O_n(\mathbb{R}) \times \O_k(\mathbb{R})$-invariance in Section~\ref{sec:St} but that only holds for the two aforementioned special cases.  The formula in \cite[Theorem~1.3]{MRW2} involves the Euler--Poincar\'{e} characteristic of a hypersurface complement, which is all but impossible to compute in situations as complex as the manifolds in this article. Similarly, the sophisticated techniques in \cite{BD2} involves representation theory and geometric invariant theory, but our calculation of Euclidean distance degrees in this article will all be based on standard linear algebra. 

\subsection*{Notations and conventions}  Although we are only interested in the flag, Grassmann, and Stiefel manifolds over $\mathbb{R}$ in this article, the nature of Euclidean distance degree requires working with their complex loci too. Our notations are chosen to reflect this field dependence. For $\mathbb{F} = \mathbb{R}$ or $\mathbb{C}$, we write $\O_n(\mathbb{F}) \coloneqq \{ X \in \mathbb{F}^{n \times n} : X^\tp X = I \}$ for the orthogonal group over $\mathbb{F}$, and $\Sym^2(\mathbb{F}^n) \coloneqq \{ X \in \mathbb{F}^{n \times n} : X^\tp = X \}$ for the subspace of  symmetric matrices in $\mathbb{F}^{n \times n}$. We stress that when $\mathbb{F} = \mathbb{C}$, we really do mean complex orthogonal and complex symmetric matrices, not unitary and Hermitian matrices; in particular, these are the complex loci of their real counterparts:
\[
\O_n(\mathbb{R})^{\mathbb{C}} = \O_n(\mathbb{C}), \qquad  \Sym^2(\mathbb{R}^n)^{\mathbb{C}} = \Sym^2(\mathbb{C}^n).
\]
We write $\mathfrak{S}_n$ for the symmetric group of permutations on $n$ elements and $\mathbb{S}^n \coloneqq \{ x \in \mathbb{R}^{n+1} : \lVert x \rVert = 1\} $ for the unit sphere in $\mathbb{R}^{n+1}$.

We write $\mathbb{T}_X \mathcal{M}$ for the tangent space at $X \in \mathcal{M}$. An embedded manifold $\mathcal{M} \subseteq \mathbb{R}^{m \times n}$ has a \emph{normal space} at $X \in \mathbb{M}$ too, denoted and given by
\[
\mathbb{N}_X \mathcal{M} \coloneqq \{ Z \in \mathbb{R}^{m \times n} : \tr(Y^\tp Z) = 0 \text{ for all } Y \in \mathbb{T}_X \mathcal{M}\}.
\]

Given that our context is that of manifold optimization, where points on manifolds are represented as matrices, we have assumed that our ambient space is that of $m \times n$ matrices $\mathbb{R}^{m \times n}$ or $n \times n$ symmetric matrices  $\Sym^2(\mathbb{R}^n)$. This is just a cosmetic difference: The more common ambient space $\mathbb{R}^m = \mathbb{R}^{m \times 1}$ is just the special case $n =1$ and conversely $\mathbb{R}^{m \times n} \cong \mathbb{R}^{mn}$ and  $\Sym^2(\mathbb{R}^n) \cong \mathbb{R}^{n(n+1)/2}$.

\section{Manifold optimization}\label{sec:man}

By their standard definitions, the flag manifold $\Flag(k_1,\dots,  k_p, \mathbb{R}^n)$ is the set of nested subspaces of dimensions $k_1 < \dots < k_p$ in $\mathbb{R}^n$, the Grassmann manifold $\Gr(k, \mathbb{R}^n)$ the set of $k$-dimensional subspaces in $\mathbb{R}^n$, and the Stiefel manifold $\V(k, \mathbb{R}^n)$ a set of orthonormal $k$-frames in $\mathbb{R}^n$. Nevertheless such an abstract set-theoretic description is of little use in manifold optimization, where it is standard practice to model these manifolds concretely as either quotient manifolds or submanifolds of matrices \cite[Tables~1 and 2]{ZLK24}. Clearly, Euclidean distance degree does not apply to the former but only to the latter and henceforth by a \emph{model} of the flag, Grassmann, or Stiefel manifolds, we mean a submanifold of $\mathbb{R}^{m \times n}$ or $\Sym^2(\mathbb{R}^n)$ diffeomorphic to it.  There is one omission from our list of manifolds --- the Cartan manifold of ellipsoids in $\mathbb{R}^n$ centered at the origin, usually modeled as the set of positive definite matrices $\Sym^2_\pp(\mathbb{R}^n)$ equipped with the Riemannian metric $g_B(X, Y) = \tr(B^{-1} X B^{-1} Y)$, $ B \in \Sym^2_\pp (\mathbb{R}^n)$ \cite[Section~8]{ZLK24}. The reason for this omission is that the notion of Euclidean distance degree does not apply since $\Sym^2_\pp (\mathbb{R}^n)$ is not an affine variety in $\Sym^2(\mathbb{R}^n)$.

It has recently been shown \cite{LK24a} that any equivariant model of the flag, Grassmann, or Stiefel manifold of minimal dimension must necessarily be one of the following:
\begin{enumerate}[\upshape (i)]
\item\label{it:iso} an \emph{isospectral model} of the flag manifold $\Flag(k_1,\dots,  k_p, \mathbb{R}^n)$ is a submanifold of $\Sym^2(\mathbb{R}^n) $ given by
\begin{align*}
\Flag_{b_1,\dots,  b_{p+1}}(k_1,\dots,  k_p, n)  &\coloneqq
\biggl\lbrace
X \in \Sym^2(\mathbb{R}^n) : \prod_{j=0}^p (X - b_{j+1} I_n) = 0,  \; \tr(X) = \sum_{j=0}^p (k_{j+1} - k_j) b_{j+1}
\biggr\rbrace \\
&= 
\left\lbrace
V \begin{bsmallmatrix}
b_1 I_{k_1} & 0 & \cdots & 0 \\
0 & b_2 I_{k_2 - k_1} & \cdots & 0 \\
\vdots & \vdots & \ddots & \vdots \\
0 & 0 & \cdots & b_{p+1} (n- k_p) \\
\end{bsmallmatrix} V^\tp \in \Sym^2(\mathbb{R}^n) : V \in \O_n(\mathbb{R})
\right\rbrace
\end{align*}
for some distinct $b_1,\dots,b_{p+1} \in \mathbb{R}$;

\item\label{it:quad} a \emph{quadratic model} of the Grassmann manifold $\Gr(k, \mathbb{R}^n)$ is a submanifold  of $\Sym^2(\mathbb{R}^n) $ given by
\begin{align*}
\Gr_{a,b}(k,n) &\coloneqq
\bigl\lbrace
X \in \Sym^2(\mathbb{R}^n)  : (X - aI) (X - bI) = 0,  \; \tr(X) = ak + b(n-k)
\bigr\rbrace \\
&= 
\biggl\lbrace
V \begin{bmatrix}
a & 0 \\
0 & b
\end{bmatrix} V^\tp \in \Sym^2(\mathbb{R}^n) : V \in \O_n(\mathbb{R})
\biggr\rbrace
\end{align*}
for some distinct $a,b \in \mathbb{R}$;

\item\label{it:chol} a \emph{Cholesky model} of the Stiefel manifold $\V(k, \mathbb{R}^n)$ is a submanifold
\[
\V_{\!B}(k,n) \coloneqq \{ X \in \mathbb{R}^{n \times k} : X^\tp X = B \}
\]
for some $B \in \Sym^2_\pp (\mathbb{R}^k)$.
\end{enumerate}
Evidently, just as the Grassmann manifold is a special case of the flag manifold with $p =1$, the quadratic model is a special case of the isospectral model with $p = 1$. We use $\Sym^2(\mathbb{R}^n)$ as our ambient space in \eqref{it:iso} and \eqref{it:quad} but it would not have made any difference to the Euclidean distance degree if we have instead used $\mathbb{R}^{n \times n}$ as ambient space: For a square matrix $A \in \mathbb{R}^{n \times n}$ and any $\mathcal{M} \subseteq \Sym^2(\mathbb{R}^n)$, the functions $\delta_A$ and $\delta_{(A + A^\tp)/2}$ only differ by a constant.

The models in \eqref{it:iso}, \eqref{it:quad}, and \eqref{it:chol} are smooth submanifolds diffeomorphic to $\Flag(k_1,\dots,  k_p, \mathbb{R}^n)$, $\Gr(k, \mathbb{R}^n)$, and $\V(k, \mathbb{R}^n)$ respectively. However, it is evident from their definitions that these submanifolds are also real affine varieties in their respective ambient spaces. In other words, the notion of Euclidean distance degree applies to all of \eqref{it:iso}, \eqref{it:quad}, and \eqref{it:chol}. The main contribution of this article is to determine them explicitly:
\begin{equation}\label{eq:main}
\begin{gathered}
\ED\bigl(\Flag_{b_1,\dots,  b_{p+1}}(k_1,\dots,  k_p, n) \bigr) = \binom{n}{k_1, k_2-k_1,\dots,  n-k_p}, \\
\ED\bigl(\Gr_{a,b}(k, n) \bigr) = \binom{n}{k}, \qquad
\ED\bigl(\V_{\!B}(k, n) \bigr) = 2^k.
\end{gathered}
\end{equation}
A noteworthy point is that these values are independent of the model parameters $b_1,\dots,  b_{p+1} \in \mathbb{R}$;  $a,b \in \mathbb{R}$; and $B \in \Sym^2_\pp (\mathbb{R}^k)$ in the respective models. As a side contribution, we define an analog of the quadratic model in \eqref{it:quad} for an important type of Schubert varieties and determine their Euclidean distance degrees. We postpone their definition and discussions to Section~\ref{sec:Gr}.

Any study of stationary points of functions on manifolds invariably invites for comparison with Morse Theory \cite{Milnor}. Indeed, the Euclidean distance function of a generic point to a manifold is the canonical example of a Morse function \cite[Theorem~6.6]{Milnor}, as we alluded to in Section~\ref{sec:ed}. But while Morse Theory applies to differential manifolds, Euclidean distance degree applies to algebraic varieties. In general they are incompatible but the models \eqref{it:iso}, \eqref{it:quad}, and \eqref{it:chol}, being simultaneously differential manifolds and algebraic varieties, provide natural examples to study their interaction. Given that there are already works \cite{AH, CDM25, MT24, MRW3} exploring the interaction between Morse Theory and Euclidean distance degree, we will not say more here.

\section{Euclidean distance degree of the flag manifold}\label{sec:Fl}

We begin with the isospectral model of the flag manifold $\Flag_{b_1,\dots,  b_{p+1}}(k_1,\dots,  k_p, n)$ in \eqref{it:iso} but we will see that its Euclidean distance degree is independent of all model parameters $b_1,\dots,b_{p+1} \in \mathbb{R}$.
\begin{proposition}\label{prop:real stationary points}
Let $A\in \Sym^2(\mathbb{R}^n)$ have distinct eigenvalues and let $A = Q \diag(a_1,\dots,  a_n) Q^\tp$ be an eigenvalue decomposition with $Q \in \O_n(\mathbb{R})$. Then  $X \in \Flag_{b_1,\dots,  b_{p+1}}(k_1,\dots,  k_p, n)$ is a stationary point of the Euclidean distance function $\delta_A$ if and only if 
\begin{equation}\label{eq:feig}
X = Q P_\tau  \diag(b_1 I_{k_1},  b_2  I_{k_2 - k_1},  \dots,  b_{p+1} I_{n - k_p}) P_\tau ^\tp Q^\tp
\end{equation}
where $P_\tau$ is the permutation matrix associated to $\tau \in \mathfrak{S}_n$. The number of distinct stationary points of $\delta_A$ is
\begin{equation}\label{eq:fcount}
\binom{n}{k_1, k_2-k_1,\dots,  n-k_p}.
\end{equation}
\end{proposition}
\begin{proof}
Since $\delta_{Q^\tp A Q}(X) = \delta_A (Q X Q^\tp)$ for any $Q\in \O_n(\mathbb{R})$, we may assume that $A = \diag(a_1,\dots, a_n)$. Observe that $X\in \Flag_{b_1,\dots,  b_{p+1}}(k_1,\dots,  k_p, n)$ is a stationary point of $\delta_A$ if and only if the Riemannian gradient of $\delta_A$ at $X$ vanishes,  which is equivalent to
\[
X - A \in \mathbb{N}_X \Flag_{b_1,\dots,  b_{p+1}}(k_1,\dots,  k_p, n).
\]
Using the second parameterization of the isospectral model in \eqref{it:iso}, we may write  $X = V D V^\tp$ with $D \coloneqq \diag(b_1 I_{k_1},  b_2  I_{k_2 - k_1},  \dots,  b_{p+1} I_{n - k_p})$.  Hence there exist $C_{j+1} \in \Sym^2( \mathbb{R}^{k_{j+1} - k_j} )$,  $j=0,\dots,p$,  such that 
\[
V D V^\tp - A = Q \begin{bmatrix}
C_1 & \cdots & 0  \\
\vdots & \ddots & \vdots \\
0 & \cdots &  C_{p+1}  \\
\end{bmatrix} Q^\tp,
\]
and so
\[
V^\tp A V  = 
\begin{bmatrix}
b_1I_{k_1} - C_1 & \cdots & 0  \\
\vdots & \ddots & \vdots \\
0 & \cdots &  b_{p+1} I_{n- k_p}-C_{p+1}  \\
\end{bmatrix}.
\]
As $D$ is invariant under conjugation by $\O_{k_1}(\mathbb{R}) \times \O_{k_2 - k_1}(\mathbb{R}) \times \dots \times \O_{n - k_p}(\mathbb{R})$, we may further assume that $C_1,\dots,C_{p+1}$ are all diagonal matrices. Given that the eigenvalues $a_1,\dots, a_n$ of $A$ are distinct, $V$ must therefore be a permutation matrix and 
\[
\begin{bmatrix}
C_1 & \cdots & 0  \\
\vdots & \ddots & \vdots \\
0 & \cdots &  C_{p+1}  \\
\end{bmatrix} = \begin{bmatrix}
b_1 I_{k_1} - \diag(a_{\tau(1)},\dots, a_{\tau(k_1)}) & \cdots  & 0 \\
\vdots & \ddots & \vdots \\
0 & \cdots & b_{p+1} I_{n-k_p} - \diag(a_{\tau(k_p+1)},\dots, a_{\tau(n)})
\end{bmatrix},
\]
where $\tau \in \mathfrak{S}_n$. It remains to count the number of such $X$'s. The image of the map 
\[
f: \mathfrak{S}_n \to  \Flag_{b_1,\dots,  b_{p+1}}(k_1,\dots,  k_p, n),\quad f(\tau) = P_\tau D P_\tau^\tp,
\] 
consists of all stationary points of $\delta_A$. Clearly $f(\tau_1) = f(\tau_2)$ if and only if $\tau_2 = \rho \tau_1 \rho^{-1}$ for some $\rho \in \prod_{j=0}^p \mathfrak{S}_{k_{j+1} - k_j}$. Hence
\[
|f(\mathfrak{S}_n)| = \biggl\lvert \mathfrak{S}_n \!\!\!\biggm/\!\!\!\biggl( \prod_{j=0}^p \mathfrak{S}_{k_{j+1} - k_j} \biggr) \biggr\rvert = \binom{n}{k_1, k_2-k_1,\dots,  n-k_p}. \qedhere
\] 
\end{proof}
As we can deduce from the proof above,  the number of stationary points of $\delta_A$ depends on the eigenvalue multiplicities of $A \in \Sym^2(\mathbb{R}^n)$.  As soon as $A$ has a repeated eigenvalue, $\delta_A$ will have an uncountably infinite number of stationary points since there can be an uncountably infinite number of possible $Q$'s in \eqref{eq:feig}. So the finite count in \eqref{eq:fcount} requires distinct eigenvalues. 

We will need the following result from \cite[p.~13]{Gantmacher59} for our next proof. The better known Autonne--Takagi factorization for diagonalizing complex symmetric matrices does not work for us as it requires $Q$ to be unitary, as opposed to complex orthogonal.
\begin{lemma}[Canonical form for complex symmetric matrices]\label{lem:canonical form}
Let $A\in \Sym^2(\mathbb{C}^n)$. Then there exists $Q\in \O_n(\mathbb{C})$ so that
\begin{equation}\label{eq:can1}
 A = Q\diag(\lambda_1 I_{k_1} + S_1, \dots, \lambda_p I_{k_p} + S_p)  Q^\tp,
\end{equation}
a block diagonal matrix with diagonal blocks of the form
\begin{equation}\label{eq:can2}
S_j = \frac{1}{2} (I_{k_j} - i J_{k_j}) T_{k_j} (I_{k_j} + i J_{k_j}) \in \Sym^2(\mathbb{C}^{k_j}),\quad j=1,\dots,p,
\end{equation}
where $I_k$ is the $k \times k$ identity matrix and
\[
J_k \coloneqq \begin{bmatrix}
0 & 0 & \dots & 0 & 1 \\
0 & 0 & \dots & 1 & 0 \\
\vdots & \vdots & \ddots & \vdots & \vdots \\
0 & 1 & \dots & 0 & 0 \\
1 & 0 & \dots & 0 & 0 
\end{bmatrix} \in \Sym^2(\mathbb{C}^k),\quad  
T_k \coloneqq  \begin{bmatrix}
0 & 1 & \dots & 0 & 0 \\
0 & 0 & \dots & 0 & 0 \\
\vdots & \vdots & \ddots & \vdots & \vdots \\
0 & 0 & \dots & 0 & 1 \\
0 & 0 & \dots & 0 & 0 
\end{bmatrix} \in \mathbb{C}^{k \times k}.
\]
Here $\lambda_1,\dots, \lambda_p$, not necessarily distinct, are the eigenvalues of $A$, and $k_1 + \dots + k_p = n$.
\end{lemma}

By Proposition~\ref{prop:real stationary points}, we may only conclude that
\[
\ED\bigl( \Flag_{b_1,\dots,  b_{p+1}}(k_1,\dots,  k_p, n) \bigr) \ge \binom{n}{k_1, k_2-k_1,\dots,  n-k_p},
\]
where we remind the reader that $\ED(\mathcal{M})$ is the number of stationary points of $\mathcal{M}^{\mathbb{C}}$, not $\mathcal{M}$. Showing equality requires a more careful argument.
\begin{theorem}\label{thm:EDdegree}
Let $\Flag_{b_1,\dots,  b_{p+1}}(k_1,\dots,  k_p, n)$ be as in \eqref{it:iso}. Then
\[
\ED\bigl( \Flag_{b_1,\dots,  b_{p+1}}(k_1,\dots,  k_p, n) \bigr) = \binom{n}{k_1, k_2-k_1,\dots,  n-k_p}.
\]
\end{theorem}
\begin{proof}
Let $A\in \Sym^2(\mathbb{R}^n)$ be generic so that eigenvalues of $A$ are all distinct. As in the proof of Proposition~\ref{prop:real stationary points} we may assume that $A = \diag(a_1,\dots, a_n)$. The complex locus of $\Flag_{b_1,\dots,  b_{p+1}}(k_1,\dots,  k_p, n)$ is
\[
\Flag_{b_1,\dots,  b_{p+1}}(k_1,\dots,  k_p, n)^{\mathbb{C}} =
\biggl\lbrace
X \in \Sym^2(\mathbb{C}^n): \prod_{j=0}^p (X - b_{j+1} I_n) = 0,  \; \tr(X) = \sum_{j=0}^p (k_{j+1} - k_j) b_{j+1}
\biggr\rbrace.
\]
The same line of argument in the proof of Proposition~\ref{prop:real stationary points} shows that $X\in  \Flag_{b_1,\dots,  b_{p+1}}(k_1,\dots,  k_p, n)^{\mathbb{C}}$ is a stationary point of $\delta_A$ if and only if $X = V D V^\tp$ for some $V\in \O_n(\mathbb{C})$ and 
\begin{equation}\label{cor:EDdegree:eq1}
V^\tp A V  = 
\begin{bmatrix}
b_1I_{k_1} - C_1 & \cdots & 0  \\
\vdots & \ddots & \vdots \\
0 & \cdots &  b_{p+1} I_{n- k_p}-C_{p+1}  \\
\end{bmatrix} \eqqcolon C
\end{equation}
for some $C_{j+1} \in \Sym^2( \mathbb{C}^{k_{j+1} - k_j})$,  $j=0,\dots,p$. If we could show that $C_1, \dots, C_{p+1}$ are diagonalizable by complex orthogonal matrices we are done; but unlike the proof of Proposition~\ref{prop:real stationary points}, this is not straightforward because  complex symmetric matrices are not always diagonalizable by complex orthogonal matrices.

Since $V \in \O_n(\mathbb{C})$, i.e., $V^\tp = V^{-1}$, \eqref{cor:EDdegree:eq1} is a similarity transformation between $A = \diag(a_1,\dots, a_n)$ and  $C = \diag(b_1 I_{k_1} - C_1,  \dots,  b_{p+1} I_{n-k_p} - C_{p+1})$. Hence the eigenvalues of $C$ are exactly $a_1,\dots, a_n$, which are all distinct; and consequently each  $b_i I_{k_i} - C_i$ must also have distinct eigenvalues, $i = 1, \dots,p+1$. As $b_i I_{k_i} - C_i$ is a complex symmetric matrix, we may apply  Lemma~\ref{lem:canonical form} to bring it into the form \eqref{eq:can1} with some complex orthogonal matrix $Q_i$. But since $b_i I_{k_i} - C_i$ has distinct eigenvalues, the matrices $S_j$'s in \eqref{eq:can2} must all be zero. In other words, each $b_i I_{k_i} - C_i$ is in fact diagonalizable by $Q_i$, $i = 1, \dots,p+1$. Hence $C$ is diagonalizable by the complex orthogonal matrix $Q \coloneqq \diag(Q_1,\dots, Q_{p+1})$.
\end{proof}
In general, it is not true that the Euclidean distance function has the same number of stationary points over $\mathcal{M}^{\mathbb{C}}$ and when restricted to $\mathcal{M}$ but from the proof of Theorem~\ref{thm:EDdegree}, we see that this holds for the flag manifold; in fact the proof gives us  the following corollary.
\begin{corollary}\label{cor:flag}
For a generic $A\in \Sym^2(\mathbb{R}^n)$, the stationary points of the Euclidean distance function $\delta_A : \Flag_{b_1,\dots,  b_{p+1}}(k_1,\dots,  k_p, n)^{\mathbb{C}} \to \mathbb{C}$ are all contained in $\Flag_{b_1,\dots,  b_{p+1}}(k_1,\dots,  k_p, n)$.
\end{corollary}

In most calculations of Euclidean distance degrees, it is generally not possible (and not necessary) to obtain all stationary points explicitly as we did in Proposition~\ref{prop:real stationary points}. But given these, it is straightforward to actually minimize the Euclidean distance function.
\begin{corollary}[Nearest point on a flag manifold]\label{cor:bestflag}
Let $A\in \Sym^2(\mathbb{R}^n)$ have distinct eigenvalues and $b_1 > \cdots  > b_p$.  Then the best approximation problem for $\Flag_{b_1,\dots,  b_{p+1}}(k_1,\dots,  k_p, n)$,
\begin{align*}
\operatorname{minimize} \quad & \lVert A - X \rVert \\
\operatorname{subject~to} \quad & X \in \Flag_{b_1,\dots,  b_{p+1}}(k_1,\dots,  k_p, n),
\end{align*}
has the unique solution
\[
X = Q  
\diag(b_1 I_{k_1},  b_2  I_{k_2 - k_1},  \dots,  b_{p+1} I_{n - k_p}) Q^\tp
\]
where $A = Q \diag(a_1,\dots,  a_n) Q^\tp$ is an eigenvalue decomposition with $a_1 > \dots > a_n$.
\end{corollary}
\begin{proof}
By Proposition~\ref{prop:real stationary points},  the best approximation problem is equivalent to 
\begin{align*}
\operatorname{minimize} \quad & \lVert \diag(a_1,\dots,  a_n) - P_\tau  \diag(b_1 I_{k_1},  b_2  I_{k_2 - k_1},  \dots,  b_{p+1} I_{n - k_p}) P_\tau ^\tp \rVert \\
\operatorname{subject~to} \quad & \tau \in \mathfrak{S}_n.
\end{align*}
Since $a_1 > \dots > a_n$ and $b_1 > \cdots > b_p$,  the unique solution is given by $\tau = \operatorname{id}$.
\end{proof}

\section{Euclidean distance degree of Grassmannian and Schubert varieties}\label{sec:Gr}

The Euclidean distance degree of the Grassmann manifold is just the special $p =1$ case of Theorem~\ref{thm:EDdegree}, but deserves stating separately because of the particular importance of Grassmannians.
\begin{corollary}\label{cor:GrEDdeg}
Let $\Gr_{a,b}(k,n)$ be as in \eqref{it:quad}. Then
\[
\ED\bigl( \Gr_{a,b}(k,n) \bigr) = \binom{n}{k}.
\]
\end{corollary}
We also have the following explicit characterization of the stationary points from Proposition~\ref{prop:real stationary points} and Corollary~\ref{cor:GrEDdeg}.
\begin{corollary}\label{cor:GrStationary}
Let $A\in \Sym^2(\mathbb{R}^n)$ have distinct eigenvalues and let $A = Q \diag(a_1,\dots,  a_n) Q^\tp$ be an eigenvalue decomposition with $Q \in \O_n(\mathbb{R})$. Then  $X \in \Gr_{a,b}(k,n)$ is a stationary point of the Euclidean distance function $\delta_A$ if and only if 
\begin{equation}\label{eq:geig}
X = Q P_\tau  \begin{bmatrix} a I_k & 0 \\ 0 & b I_{n-k} \end{bmatrix} P_\tau ^\tp Q^\tp
\end{equation}
where $P_\tau$ is the permutation matrix associated to $\tau \in \mathfrak{S}_n$. The $\binom{n}{k}$ distinct stationary points of $\delta_A : \Gr_{a,b}(k,n)^{\mathbb{C}} \to \mathbb{C}$ are all contained in $\Gr_{a,b}(k,n)$.
\end{corollary}

Our goal in this section is not so much to state Corollaries~\ref{cor:GrEDdeg} and \ref{cor:GrStationary} but to find the Euclidean distance degrees of a special submanifold of the Grassmannian so important it is given its own name: the Schubert variety. It has played a central role in the topology and geometry of the Grassmannian in the last 150 years \cite{Fulton97, GH94, Kleiman76, Manivel01, Schubert79} and has recently been shown to be important in applications too  \cite{KMK24, LZFYY24, KN23, YL16}.

Let $k\le l \le m \le n$ be positive integers throughout this section. We will use blackboard bold $\mathbb{U}$, $\mathbb{V}$, $\mathbb{W}$  to denote subspaces, i.e., points on various Grassmannians.
\begin{definition}\label{def:Schu}
Let $\mathbb{U} \in \Gr(k,\mathbb{R}^n)$ and  $\mathbb{W} \in \Gr(m,\mathbb{R}^n)$, i.e., $\mathbb{U}$  and $\mathbb{W}$ are respectively a $k$- and an $m$-dimensional subspace in $\mathbb{R}^n$. 
Then the Schubert variety of $l$-dimensional subspaces between $\mathbb{U}$ and $\mathbb{W}$ is denoted and defined by
\begin{equation}\label{eq:Om}
\Omega(\mathbb{U},  \mathbb{W}) \coloneqq \lbrace 
\mathbb{V} \in \Gr(l,\mathbb{R}^n):   \mathbb{U} \subseteq  \mathbb{V} \subseteq \mathbb{W}
\rbrace.
\end{equation}
\end{definition}
The set $\Omega(\mathbb{U},  \mathbb{W})$ is both an algebraic subvariety and a smooth submanifold of the Grassmannian $\Gr(l,\mathbb{R}^n)$.
It is also a Schubert variety in the sense of \cite{Fulton97, GH94, Kleiman76, Manivel01, Schubert79} but we caution the reader that not every Schubert variety can be written in this form. In fact this Schubert variety is quite special in that it is isomorphic to a Grassmannian:
\begin{equation}\label{eq:OmGr}
\Omega(\mathbb{U},  \mathbb{W}) = \{ \mathbb{V} \in \Gr(l,  \mathbb{R}^n): \mathbb{V}/\mathbb{U} \subseteq \mathbb{W}/\mathbb{U} \} \cong \Gr(l-k,  \mathbb{W}/\mathbb{U}) \cong \Gr(l-k,  \mathbb{R}^{m - k}).
\end{equation}
The special cases $\Omega(\mathbb{U}, \mathbb{R}^n)$ and $\Omega( \{0\},  \mathbb{W})$ give us the Schubert varieties studied in \cite{YL16}:
\[
\Omega_\p (\mathbb{U}) \coloneqq \lbrace 
\mathbb{V} \in \Gr(m,\mathbb{R}^n):  \mathbb{U} \subseteq \mathbb{V}
\rbrace, \qquad \Omega_\m (\mathbb{W}) \coloneqq \lbrace 
\mathbb{V} \in \Gr(k,\mathbb{R}^n):  \mathbb{V} \subseteq \mathbb{W}
\rbrace
\]
are the Schubert variety of $m$-dimensional subspaces containing $\mathbb{U}$ and the Schubert variety of $k$-dimensional subspaces contained in $\mathbb{W}$.  

For distinct real numbers $a, b$ and positive integers $l \le n$, the map
\[
\varphi : \Gr(l,\mathbb{R}^n) \to \Gr_{a,b}(l,n),\quad \varphi (\mathbb{V}) = Q \diag(aI_l, b I_{n-l}) Q^\tp,
\]
where the first $l$ column vectors of $Q \in \O_n(\mathbb{R})$ form an orthonormal basis of $\mathbb{V}$, gives an isomorphism $ \Gr(l,\mathbb{R}^n) \cong \Gr_{a,b}(l,n)$ \cite{LK24a}. Its inverse is given by $\varphi^{-1}(X) = \spn\{q_1,\dots,  q_l\}$,  where $X = Q I_{l,n-l} Q^\tp$ and $q_1,\dots,  q_l$ are the first $l$ column vectors of $Q\in \O_n(\mathbb{R})$. The existence of this isomorphism $\varphi$, shown to be equivariant and has minimal ambient dimension in \cite{LK24a}, is the justification for calling its image $\Gr_{a,b}(l,n) = \varphi\bigl(\Gr(l,\mathbb{R}^n)\bigr)$ the quadratic model  of $ \Gr(l,\mathbb{R}^n)$. We will show that $\varphi\bigl(\Omega(\mathbb{U},  \mathbb{W}) \bigr)$ gives us a neat model for the Schubert variety of $l$-dimensional subspaces between $\mathbb{U}$ and $\mathbb{W}$ too.
\begin{theorem}[Quadratic models for Schubert varieties]\label{thm:Schubert}
Let $k \le l \le m \le n$ be positive integers and $\mathbb{U} \in \Gr(k,\mathbb{R}^n)$, $\mathbb{W} \in \Gr(m, \mathbb{R}^n)$. Let $Q = [q_1,\dots,  q_n] \in \O_n(\mathbb{R})$ be such that $\spn\{q_1,\dots,  q_k\} = \mathbb{U}$ and $\spn\{q_{k+1},\dots,  q_{k+n-m}\} = \mathbb{W}^{\perp}$.  Then the image of $\varphi$ restricted to $\Omega(\mathbb{U},\mathbb{W})$ is given by
\begin{equation}\label{eq:OmAB}
\Omega_{a,b}(\mathbb{U},\mathbb{W}) \coloneqq \left\lbrace 
Q \begin{bmatrix}
aI_k & 0 &0\\
0 & bI_{n-m} &0\\
0 & 0 & X
\end{bmatrix} Q^\tp \in \Gr_{a,b}(l, n) : X \in \Gr_{a,b}(l-k, m - k)
\right\rbrace.
\end{equation}
\end{theorem}
\begin{proof}
We need to show that $\varphi\bigl(\Omega(\mathbb{U},  \mathbb{W})\bigr) = \Omega_{a,b}(\mathbb{U},\mathbb{W})$. By \eqref{it:quad}, each $X \in \Gr_{a,b}(l-k,m-k)$ takes the form $X = V \diag(a I_{l-k},  b I_{m-k}) V^\tp$ for some $V \in \O_{m-k}(\mathbb{R})$.  So
\[
Q \begin{bmatrix}
aI_k & 0 &0\\
0 & bI_{n-m} &0\\
0 & 0 & X
\end{bmatrix} Q^\tp = Q \begin{bmatrix}
I_{n + k - m} & 0 \\
0 & V
\end{bmatrix} \begin{bmatrix}
aI_k & 0 & 0 & 0\\
0 & b I_{n-m} & 0 & 0\\
0 & 0 & a I_{l-k} & 0  \\
0 & 0 & 0 & b I_{m-l} 
\end{bmatrix} \begin{bmatrix}
I_{n + k - m} & 0 \\
0 & V^\tp
\end{bmatrix} Q^\tp \eqqcolon Y.
\]
Let $Q \begin{bsmallmatrix}
I_{n+k - m} & 0 \\
0 & V
\end{bsmallmatrix}  = [q_1,\dots,  q_n ]$. Then
\[
\varphi^{-1}(Y) = \spn \{q_1,\dots, q_k,  q_{n+k-m+1},\dots,  q_{n + l - m} \}.
\]
Note that $\varphi^{-1}(Y)$ is $l$-dimensional, $\mathbb{U}  \subseteq \varphi^{-1}(Y)$, and $\mathbb{W}^{\perp} \subseteq \varphi^{-1}(Y)^{\perp}$. Hence we have shown that $\varphi^{-1} \bigl( \Omega_{a,b}(\mathbb{U}, \mathbb{W}) \bigr) \subseteq \Omega(\mathbb{U},\mathbb{W})$. As $\Omega_{a,b}(\mathbb{U},\mathbb{W})$ is Zariski closed and $\varphi$ is a regular map, $\varphi^{-1}\bigl(\Omega_{a,b}(\mathbb{U},\mathbb{W})\bigr)$ is also Zariski closed.  By \eqref{eq:OmGr},
\[
\dim \Omega(\mathbb{U},\mathbb{W}) = (l-k)(m-l) = \dim \Gr_{a,b}(l-k,m-k) = \dim \Omega_{a,b}(\mathbb{U}, \mathbb{W}),
\]
and thus we in fact have equality.
\end{proof}

We will call \eqref{eq:OmAB} the quadratic model of the Schubert variety $\Omega_{a,b}(\mathbb{U},\mathbb{W})$. We will now deduce its Euclidean distance degree.
\begin{corollary}\label{cor:ED Omega}
Let $\Omega_{a,b}(\mathbb{U},\mathbb{W})$ be as in \eqref{eq:OmAB}.  Then
\[
\ED \bigl( \Omega_{a,b} (\mathbb{U},  \mathbb{W}) \bigr) = \binom{m-k}{l-k}.
\]
\end{corollary}
\begin{proof}
For any $A \in \mathsf{S}^2(\mathbb{R}^n)$ and $Y \in \Omega_{a,b}(\mathbb{U},\mathbb{W})$,
\[
\delta_A(Y) = \delta_{Q^\tp A Q} \left( \begin{bmatrix}
aI_k & 0 &0\\
0 & bI_{n-m} &0\\
0 & 0 & X
\end{bmatrix} 
\right). 
\]
So we may assume that $Q = I_n$ and thus
\begin{equation}\label{eq:Oab}
\Omega_{a,b}(\mathbb{U},\mathbb{W}) = \left\lbrace 
\begin{bmatrix}
aI_k & 0 &0\\
0 & bI_{n-m} &0\\
0 & 0 & X
\end{bmatrix} \in \Gr_{a,b}(l, n) :  X \in \Gr_{a,b}(l-k, m - k)
\right\rbrace.
\end{equation}
The ambient space of symmetric matrices decomposes as $3 \times 3$ block matrices into
\[
 \mathsf{S}^{2}(\mathbb{R}^n) \cong  \mathsf{S}^{2}(\mathbb{R}^k) \oplus \mathsf{S}^{2}(\mathbb{R}^{n-m}) \oplus \mathsf{S}^{2}(\mathbb{R}^{m-k}) \oplus \bigl[ \mathbb{R}^{k \times (n-m)} \oplus \mathbb{R}^{k \times (m-k)} \oplus \mathbb{R}^{(n-m) \times (m-k)} \bigr]
\]
where the terms in the bracket are the spaces of the $(1,2)$, $(1,3)$, $(2,3)$ blocks. This gives a corresponding decomposition
\[
\Omega_{a,b} (\mathbb{U},\mathbb{W})  \cong  \{a I_k\} \times \{b I_{n-m}\} \times \Gr_{a,b}(l-k,  m -k) \times  \{0 \}
\]
where the last $\{0\}$ represents the zeros in the $(1,2)$, $(1,3)$, $(2,3)$ blocks in \eqref{eq:Oab} collectively.
For any $A \in \mathsf{S}^2(\mathbb{R}^n)$ partitioned into $3 \times 3$ blocks, we also have
\[
\left\lVert \begin{bmatrix} 
A_{11} & A_{12} & A_{13}\\ 
A_{12}^\tp & A_{22}  &  A_{23} \\
A_{13}^\tp & A_{23}^\tp  &  A_{33}
\end{bmatrix} \right\rVert^2 
= \lVert A_{11} \rVert^2 + 2 \lVert A_{12} \rVert^2 + 2\lVert A_{13} \rVert^2 + 
\lVert A_{22} \rVert^2 + 2 \lVert A_{23} \rVert^2 + \lVert A_{33} \rVert^2
\]
and thus
\[
\delta_A(Y) = c + \delta_{A_{33}}(X) 
\]
for some constant $c$. Hence $\ED\bigl(\Omega_{a,b} (\mathbb{U},  \mathbb{W})\bigr) = \ED \bigl(\Gr_{a,b}(l-k, m -k) \bigr) = \binom{m-k}{l-k}$ by Corollary~\ref{cor:GrEDdeg}. 
\end{proof}

We may also obtain an explicit form for the stationary points as before. In the following proposition, all quantities are as in Theorem~\ref{thm:Schubert}.
\begin{proposition}\label{thm:Schubert}
Let $A \in \mathsf{S}^2(\mathbb{R}^n)$ be such that
\[
Q^\tp A Q = \begin{bmatrix} 
B_{11} & B_{12} & B_{13}\\ 
B_{12}^\tp & B_{22}  &  B_{23} \\
B_{13}^\tp & B_{23}^\tp  &  B_{33}
\end{bmatrix} 
\]
and $B_{33} \in \mathbb{R}^{(m-l) \times (m-l)}$ has distinct eigenvalues. Let $B_{33} = V D V^\tp$, $V \in \O_{m-l}(\mathbb{R})$ be the unique eigenvalue decomposition. Then $Y \in \Omega_{a,b} (\mathbb{U},\mathbb{W})$ is a stationary point of $\delta_A$ if and only if 
\[
Y = Q \begin{bmatrix}
I_{n + k -m} & 0 \\
0 & V P_\tau 
\end{bmatrix} \begin{bmatrix}
aI_k & 0 & 0 & 0\\
0 & b I_{n-m} & 0 & 0\\
0 & 0 & a I_{l-k} & 0  \\
0 & 0 & 0 & b I_{m-l} 
\end{bmatrix} 
\begin{bmatrix}
I_{n+k -m} & 0 \\
0 & P_\tau ^\tp V^\tp
\end{bmatrix} Q^\tp,
\]
where $P_\tau $ is the permutation matrix associated with $\tau \in \mathfrak{S}_{m- l}$. The $\binom{m-k}{l-k}$ distinct stationary points of $\delta_A : \Omega_{a,b} (\mathbb{U},\mathbb{W})^{\mathbb{C}} \to \mathbb{C}$ are all contained in $\Omega_{a,b} (\mathbb{U},\mathbb{W})$.
\end{proposition}
\begin{proof}
As in the proof of Corollary~\ref{cor:ED Omega},  we may assume $Q = I_n$ so that $
\Omega_{a,b} (\mathbb{U},\mathbb{W}) \cong \{a I_k\} \times \{b I_{n-m}\} \times \Gr_{a,b}(l-k,  m -k) \times  \{0 \}$.  Thus a stationary point of $\delta_A$ on $\Omega_{a,b}(\mathbb{U},\mathbb{W})$ must take the form $\diag(a I_k,  bI_{n-m},  X)$ where $X \in \Gr_{a,b}(l-k,  m -k)$ is a stationary point of $\delta_{B_{33}} : \Gr_{a,b}(l-k,  m -k) \to \mathbb{R}$. Now we simply invoke Corollary~\ref{cor:GrStationary} to see that $X$ is given by \eqref{eq:geig}.
\end{proof}

For completeness, we state the following analog of Corollary~\ref{cor:bestflag}. Note that it includes the nearest point on a Grassmannian as the special case $\Omega_{a,b}(\{ 0\}, \mathbb{R}^n) = \Gr_{a,b}(k,n)$. Again, all quantities are as in Theorem~\ref{thm:Schubert}.
\begin{corollary}[Nearest point on a Schubert variety]\label{cor:bestGr}
Let $A\in \Sym^2(\mathbb{R}^n)$ as in Proposition~\ref{thm:Schubert} and $a > b$. Then the best approximation problem for $\Omega_{a,b}(\mathbb{U}, \mathbb{W})$,
\begin{align*}
\operatorname{minimize} \quad & \lVert A - X \rVert \\
\operatorname{subject~to} \quad & X \in \Omega_{a,b}(\mathbb{U}, \mathbb{W}),
\end{align*}
has the unique solution
\[
Q \begin{bmatrix}
I_{n + k -m} & 0 \\
0 & V 
\end{bmatrix} \begin{bmatrix}
aI_k & 0 & 0 & 0\\
0 & b I_{n-m} & 0 & 0\\
0 & 0 & a I_{l-k} & 0  \\
0 & 0 & 0 & b I_{m-l} 
\end{bmatrix} 
\begin{bmatrix}
I_{n+k -m} & 0 \\
0 & V^\tp
\end{bmatrix} Q^\tp
\]
where $Q, V$ are as in Proposition~\ref{thm:Schubert} except that we choose $V$ so that the eigenvalue decomposition $B_{33} = V \diag(\lambda_1,\dots,  \lambda_{m-l}) V^\tp$ has eigenvalues in decreasing order $\lambda_1 > \cdots > \lambda_{m-l}$.
\end{corollary}
\begin{proof}
As in the proof of Proposition~\ref{thm:Schubert}, the best approximation problem is equivalent to 
\begin{align*}
\operatorname{minimize} \quad & \lVert \diag(\lambda_1,\dots,  \lambda_{m-l}) - P_{\tau} \diag(a I_{l-k},  b_{m-l}) P_{\tau}^\tp \rVert \\
\operatorname{subject~to} \quad & \tau \in \mathfrak{S}_{m-k}.
\end{align*}
Since $\lambda_1 > \cdots > \lambda_{m-l}$ and $a > b$, the unique solution is given by $\tau = \operatorname{id}$.
\end{proof}
It is easy to check that the solution for $p=1$ in Corollary~\ref{cor:bestflag} agrees with the solution for $\mathbb{U} = \{0\}$, $\mathbb{W} = \mathbb{R}^n$ in Corollary~\ref{cor:bestGr} since both reduces to $\Gr_{a,b}(k,n)$.

\section{Euclidean distance degree of Stiefel manifold}\label{sec:St}

We now consider the Cholesky model of the Stiefel manifold $\V_{\!B}(k,n)$ in \eqref{it:chol}. Again we will see that its Euclidean distance degree is independent of the model parameter $B \in \Sym^2_\pp (\mathbb{R}^k)$.
\begin{proposition}\label{prop:Stiefel real stationary points}
Let  $B = Q  \diag(b_1,  \dots,  b_k) Q^\tp $ be the eigenvalue decomposition of $B \in \Sym^2_\pp (\mathbb{R}^k)$ with $Q \in \O_k(\mathbb{R})$ and $b_1 ,\dots, b_k > 0$. Let $A \in \mathbb{R}^{n \times k}$ be generic and
\[
A Q  = U \begin{bmatrix}
\Sigma \\
0
\end{bmatrix} V^\tp, \quad U \in \O_n(\mathbb{R}),  \quad V \in \O_k(\mathbb{R}), \quad \Sigma = \diag(a_1,\dots,  a_k),
\]
be its singular value decomposition. Then $X \in \V_{\!B}(k,n)$ is a stationary point of the Euclidean distance function $\delta_A$ if and only if 
\begin{equation}\label{eq:seig}
X = U 
 \begin{bmatrix}
\varepsilon_1 \sqrt{b_1} & \cdots & 0 \\
\vdots & \ddots & \vdots \\
0 & \cdots & \varepsilon_k \sqrt{b_k} \\
0 & \cdots  & 0 \\
\vdots & \ddots & \vdots \\
0 & \cdots & 0
\end{bmatrix}
V^\tp Q^\tp ,\qquad \varepsilon_1,\dots,  \varepsilon_k \in \{-1,1\}.
\end{equation}
The number of distinct stationary points of $\delta_A$ is $2^k$.
\end{proposition}
\begin{proof}
Let  $D = \diag(b_1,  \dots,  b_k)$. Since
\[
\V_{\!B} (k,n) = \V_{Q D Q^\tp }(k,n) =  \V_D(k,n) Q^\tp , \qquad \delta_{U^\tp A V} (X) = \delta_A (U X V^\tp)
\]
for $Q, V \in \O_k(\mathbb{R})$ and $U \in \O_n(\mathbb{R})$,  we may assume that
\begin{equation}\label{eq:AB}
A = \begin{bmatrix}
\Sigma \\
0 
\end{bmatrix} =  \begin{bmatrix}
a_1 & \cdots & 0 \\
\vdots & \ddots & \vdots \\
0 & \cdots & a_k \\
0 & \cdots  & 0 \\
\vdots & \ddots & \vdots \\
0 & \cdots & 0
\end{bmatrix}
, \qquad B = D = \begin{bmatrix}
b_1 & \cdots & 0 \\
\vdots & \ddots & \vdots \\
0 & \cdots & b_k 
\end{bmatrix}.
\end{equation}
Since $A$ is generic, we may further assume that $a_1 > \dots > a_k > 0$ and $\sqrt{b_i} a_i \ne \sqrt{\smash[b]{b_j}} a_j$ for all $i \ne j$, $i,j \in \{1,\dots, k\}$.

The tangent space at $X \in \V_{\!B}(k,n)$ is
\[
\mathbb{T}_X \V_{\!B}(k,n) = \{Y \in \mathbb{R}^{n \times k} : X^\tp Y + Y^\tp X = 0 \}
\]
and so the normal space is
\[
\mathbb{N}_X \V_{\!B}(k,n) = \{X Z \in \mathbb{R}^{n\times k} : Z \in \Sym^2( \mathbb{R}^k) \}.
\]
Hence $X \in \V_{\!B} (k,n)$ is a stationary point of $\delta_A$ if and only if $A - X  = XZ$ for some $Z \in \Sym^2(\mathbb{R}^k)$.  If we partition
\[
X = \begin{bmatrix}
X_1 \\
X_2 
\end{bmatrix}, \quad X_1 \in \mathbb{R}^{k \times k}, \quad X_2 \in \mathbb{R}^{(n-k) \times k},
\]
then
\[
\begin{bmatrix}
\Sigma \\ 
0
\end{bmatrix} = \begin{bmatrix}
X_1 (Z+I_k) \\
X_2 (Z + I_k)
\end{bmatrix},
\]
from which we obtain $X_2(Z + I_k) = 0$ and $\Sigma = X_1 (Z + I_k)$.  Since $\Sigma$ is invertible,  both $X_1$ and $Z + I_k$ must also be invertible, implying that $X_2 = 0$ and $X_1 = \Sigma (Z + I_k)^{-1}$.  Since $X \in \V_{\!B}(k,n)$, we have
\[
B = X^\tp X = X_1^\tp X_1 = (Z + I_k)^{-1} \Sigma^2  (Z + I_k)^{-1},
\]
and so
\[
I_k = [\Sigma  (Z + I_k)^{-1} B^{-1/2}  ]^\tp [ \Sigma  (Z + I_k)^{-1} B^{-1/2}  ].
\]
It follows that $  (Z + I_k)^{-1} = \Sigma^{-1} W  B^{1/2}$ for some $W \in \O_k(\mathbb{R})$.  As $Z + I_k \in \Sym^2( \mathbb{R}^k)$,  we must have $C^{-1} W  C = W^\tp$ where $C \coloneqq B^{1/2} \Sigma = \diag(\sqrt{b_1} a_1,\dots,  \sqrt{b_k} a_k)$. Write $c_i = \sqrt{b_i} a_i$,  $i=1,\dots, k$, where we may also assume $c_1 > \dots > c_k$ without loss of generality.  Since
\[
W = W^{-\tp} = C^{-1} W^{-1} C = C^{-1} W^\tp  C = C^{-2} W C^2,
\]
we obtain $(1 - c_i^{-2} c_j^2) w_{ij} = 0$ for any $ i,  j = 1,\dots, k$, implying that $w_{ij} = 0$ for all $i \ne j$ and $w_{ii}^2 = 1$ for all $i$.  Hence $X_1 = W B^{1/2} = \diag(\pm \sqrt{b_1},\dots,  \pm \sqrt{b_k})$, as required for \eqref{eq:seig}.
\end{proof}
We may now deduce the required Euclidean distance degree. For the flag manifold, the proof of Proposition~\ref{prop:real stationary points} for real stationary points requires some tweaking for the complex case in Theorem~\ref{thm:EDdegree}. Fortunately for the Stiefel manifold, the proof of  Proposition~\ref{prop:Stiefel real stationary points}  applies verbatim to the complex case.
\begin{theorem}\label{thm:Stiefl EDdegree}
Let $\V_{\!B}(k,n)$  be as in \eqref{it:chol}. Then
\[
\ED\bigl(\V_{\!B}(k,n)\bigr) = 2^k.
\]
\end{theorem}
\begin{proof}
Let $A \in \mathbb{R}^{n \times k}$ be generic. We now repeat the proof of Proposition~\ref{prop:Stiefel real stationary points}   with  the complex locus
\[
\V_{\!B}(k,n)^{\mathbb{C}} \coloneqq \{ X \in \mathbb{C}^{n \times k} : X^\tp X = B \}
\]
in place of  $\V_{\!B}(k,n)$. 
Note that $A$ and $B$ both remain real matrices as in the proof of Proposition~\ref{prop:Stiefel real stationary points}. The only difference here is that $X$ is allowed to be a complex matrix. So we may assume that $A$ and $B$ have the forms in \eqref{eq:AB}. It remains to observe that the rest of the proof of Proposition~\ref{prop:Stiefel real stationary points} carries over verbatim with $\mathbb{C}$ in place of $\mathbb{R}$, leading to the required conclusion that $\delta_A$ has exactly $2^k$ stationary points in $\V_{\!B}(k,n)^{\mathbb{C}}$.
\end{proof}

As we mentioned in Section~\ref{sec:ed} the special case with $n = k$ and $B= I$ gives $\V_I (n,n) = \O_n(\mathbb{R})$, which is a classical example in \cite[Example~4.12]{DLO17}. Theorem~\ref{thm:Stiefl EDdegree} confirms that $\ED(\O_n(\mathbb{R})) = 2^n$. Another special case with $k =1$ and  $B=1$ gives  $\V_1 (1,n) = \mathbb{S}^{n-1}$, a hypersurface in $\mathbb{R}^n$ and thus \cite[Corollary~2.12]{Lee17} applies to give $\ED(\mathbb{S}^{n-1}) = 2$, which agrees with the value given by Theorem~\ref{thm:Stiefl EDdegree}.

A comparison of Proposition~\ref{prop:Stiefel real stationary points} and Theorem~\ref{thm:Stiefl EDdegree} yields the following analogue of  Corollary~\ref{cor:flag}.
\begin{corollary}\label{cor:st}
For a generic $A\in \mathbb{R}^{n \times k}$, the stationary points of the Euclidean distance function $\delta_A : \V_{\!B}(k,n)^{\mathbb{C}} \to \mathbb{C}$ are all contained in $\V_{\!B}(k,n)$.
\end{corollary}

Lastly, we provide the analog of Corollaries~\ref{cor:bestflag} and \ref{cor:bestGr}.
\begin{corollary}[Nearest point on a Stiefel manifold]\label{cor:bestSt}
Let $A\in  \mathbb{R}^{n \times k}$ be generic and $b_1 > \cdots > b_k > 0$. Then the best approximation problem for $\V_{\!B}(k,n)$,
\begin{align*}
\operatorname{minimize} \quad & \lVert A - X \rVert \\
\operatorname{subject~to} \quad & X \in \V_{\!B}(k,n)
\end{align*}
has the unique solution
\[
U 
 \begin{bmatrix}
\sqrt{b_1} & \cdots & 0 \\
\vdots & \ddots & \vdots \\
0 & \cdots & \sqrt{b_k} \\
0 & \cdots  & 0 \\
\vdots & \ddots & \vdots \\
0 & \cdots & 0
\end{bmatrix}
V^\tp Q^\tp
\]
where $U,V,Q$ are as in Proposition~\ref{prop:Stiefel real stationary points}.
\end{corollary}
\begin{proof}
By Proposition~\ref{prop:Stiefel real stationary points}, the best approximation problem is equivalent to 
\begin{align*}
\operatorname{minimize} \quad & \lVert \diag\! \begin{pmatrix} a_1,\dots,  a_k\end{pmatrix} - \diag\!\begin{pmatrix}\varepsilon_1 \sqrt{b_1},\dots,  \varepsilon_k \sqrt{b_k}\end{pmatrix}\rVert \\
\operatorname{subject~to} \quad & \varepsilon_1,\dots,  \varepsilon_k \in \{-1,1\}.
\end{align*}
Since $a_1 > \cdots > a_k$ and $b_1 > \cdots > b_k > 0$, the unique solution is given by $\varepsilon_1 = \dots = \varepsilon_k = 1$.
\end{proof}

\section{Conclusion}

We hope that these results on Euclidean distance degrees would shed further light on the computational complexity of manifold optimization problems, complementing the study in \cite{ZLK24}.  Although there is as yet no definitive link between the tractability of general optimization problems over $\mathcal{M}$ and $\ED(\mathcal{M})$, evidence suggests that the latter reflects some facets of the former. So the revelation that the Euclidean distance degrees of the flag, Grassmann, and Stiefel manifolds are multinomial coefficients and powers of two---generally large numbers---may be taken as another indication that optimization problems over these manifolds are difficult, attesting to the NP-hardness results in \cite{ZLK24}. We caution the reader that this does not mean that global minimization of the Euclidean distance function $\delta_A$ for a generic $A$ is computationally intractable. As Corollaries~\ref{cor:bestflag},  \ref{cor:bestGr}, and \ref{cor:bestSt} show, for the manifolds considered in this article, their global minimizers of $\delta_A$ have simple closed-form expressions computable to arbitrary accuracy in polynomial time.  On the other hand, any linear manifold $\mathcal{M}$ automatically has $\ED(\mathcal{M}) = 1$ but the global minimization of $\delta_A$ for the ``two-sided product'' manifold $\mathcal{M} = \{X \in \mathbb{R}^{m \times n} : F X G = H \} $ for fixed $F \in \mathbb{R}^{p \times m}$, $G \in \mathbb{R}^{n \times q}$, $H \in \mathbb{R}^{p \times q}$ is considerably more involved \cite{LL07,  LL23}.

We end this article with two open problems for the readers: Determine the Euclidean distance degree for the symplectic Grassmannian in \cite[Equation 20]{LXK25} and for a general Schubert variety that does not take the form in Definition~\ref{def:Schu}. The former is already expressed in \cite[Equation 20]{LXK25} as a submanifold embedded in $\mathbb{R}^{n \times n}$; for the latter, one would need to first find a quadratic model for an arbitrary Schubert variety similar to \eqref{eq:OmAB}.

\bibliographystyle{abbrv}

\begin{thebibliography}{10}

\bibitem{AMS}
P.-A. Absil, R.~Mahony, and R.~Sepulchre.
\newblock {\em Optimization algorithms on matrix manifolds}.
\newblock Princeton University Press, Princeton, NJ, 2008.

\bibitem{AH}
P.~Aluffi and C.~Harris.
\newblock The {E}uclidean distance degree of smooth complex projective
  varieties.
\newblock {\em Algebra Number Theory}, 12(8):2005--2032, 2018.

\bibitem{BD1}
J.~A. Baaijens and J.~Draisma.
\newblock Euclidean distance degrees of real algebraic groups.
\newblock {\em Linear Algebra Appl.}, 467:174--187, 2015.

\bibitem{BD2}
A.~Bik and J.~Draisma.
\newblock A note on {ED} degrees of group-stable subvarieties in polar
  representations.
\newblock {\em Israel J. Math.}, 228(1):353--377, 2018.

\bibitem{BSW}
P.~Breiding, F.~Sottile, and J.~Woodcock.
\newblock Euclidean distance degree and mixed volume.
\newblock {\em Found. Comput. Math.}, 22(6):1743--1765, 2022.

\bibitem{optim1}
P.~Breiding and N.~Vannieuwenhoven.
\newblock The condition number of {R}iemannian approximation problems.
\newblock {\em SIAM J. Optim.}, 31(1):1049--1077, 2021.

\bibitem{optim2}
D.~Cifuentes, C.~Harris, and B.~Sturmfels.
\newblock The geometry of {SDP}-exactness in quadratic optimization.
\newblock {\em Math. Program.}, 182(1-2):399--428, 2020.

\bibitem{optim3}
C.~Ding, D.~Sun, J.~Sun, and K.-C. Toh.
\newblock Spectral operators of matrices: semismoothness and characterizations
  of the generalized {J}acobian.
\newblock {\em SIAM J. Optim.}, 30(1):630--659, 2020.

\bibitem{DHOST14}
J.~Draisma, E.~Horobe\c{t}, G.~Ottaviani, B.~Sturmfels, and R.~Thomas.
\newblock The {E}uclidean distance degree.
\newblock In {\em S{NC} 2014---{P}roceedings of the 2014 {S}ymposium on
  {S}ymbolic-{N}umeric {C}omputation}, pages 9--16. ACM, New York, 2014.

\bibitem{DHOST16}
J.~Draisma, E.~Horobe\c{t}, G.~Ottaviani, B.~Sturmfels, and R.~R. Thomas.
\newblock The {E}uclidean distance degree of an algebraic variety.
\newblock {\em Found. Comput. Math.}, 16(1):99--149, 2016.

\bibitem{DLO17}
D.~Drusvyatskiy, H.-L. Lee, G.~Ottaviani, and R.~R. Thomas.
\newblock The {E}uclidean distance degree of orthogonally invariant matrix
  varieties.
\newblock {\em Israel J. Math.}, 221(1):291--316, 2017.

\bibitem{EAS}
A.~Edelman, T.~A. Arias, and S.~T. Smith.
\newblock The geometry of algorithms with orthogonality constraints.
\newblock {\em SIAM J. Matrix Anal. Appl.}, 20(2):303--353, 1999.

\bibitem{Fulton97}
W.~Fulton.
\newblock {\em Young tableaux}, volume~35 of {\em London Mathematical Society
  Student Texts}.
\newblock Cambridge University Press, Cambridge, 1997.

\bibitem{Gantmacher59}
F.~R. Gantmacher.
\newblock {\em Applications of the theory of matrices}.
\newblock Interscience Publishers, New York, NY, 1959.

\bibitem{GH94}
P.~Griffiths and J.~Harris.
\newblock {\em Principles of algebraic geometry}.
\newblock Wiley Classics Library. John Wiley \& Sons, Inc., New York, 1994.

\bibitem{optim4}
D.~Henrion, S.~Naldi, and M.~Safey El~Din.
\newblock Exact algorithms for linear matrix inequalities.
\newblock {\em SIAM J. Optim.}, 26(4):2512--2539, 2016.

\bibitem{CDM25}
C.~Joita, D.~Siersma, and M.~Tibar.
\newblock Euclidean distance discriminants and {M}orse attractors, 2025.

\bibitem{KMK24}
A.~Khanna et~al.
\newblock Single-neuronal elements of speech production in humans.
\newblock {\em Nature}, 626:603–610, 2024.

\bibitem{Kleiman76}
S.~L. Kleiman.
\newblock Problem 15: rigorous foundation of {S}chubert's enumerative calculus.
\newblock In {\em Mathematical developments arising from {H}ilbert problems},
  volume Vol. XXVIII of {\em Proc. Sympos. Pure Math.}, pages 445--482. Amer.
  Math. Soc., Providence, RI, 1976.

\bibitem{ZLK24}
Z.~Lai, L.-H. Lim, and K.~Ye.
\newblock Grassmannian optimization is {NP}-hard.
\newblock {\em arXiv:2406.19377}, 2024.

\bibitem{Lee17}
H.~Lee.
\newblock The {E}uclidean distance degree of {F}ermat hypersurfaces.
\newblock {\em J. Symbolic Comput.}, 80:502--510, 2017.

\bibitem{LZFYY24}
X.~Li, T.~Zhou, X.~Feng, S.-T. Yau, and S.~S.-T. Yau.
\newblock Exploring geometry of genome space via {G}rassmann manifolds.
\newblock {\em Innovation}, 5(5):100677, 2024.

\bibitem{LL23}
Z.~Li and L.-H. Lim.
\newblock Generalized matrix nearness problems.
\newblock {\em SIAM J. Matrix Anal. Appl.}, 44(4):1709--1730, 2023.

\bibitem{LL07}
A.~Liao and Y.~Lei.
\newblock Optimal approximate solution of the matrix equation {$AXB=C$} over
  symmetric matrices.
\newblock {\em J. Comput. Math.}, 25(5):543--552, 2007.

\bibitem{LXK25}
L.-H. Lim, X.~Lu, and K.~Ye.
\newblock Special orthogonal, special unitary, and symplectic groups as
  products of {G}rassmannians.
\newblock {\em arXiv:2501.18172}, 2025.

\bibitem{LK24a}
L.-H. Lim and K.~Ye.
\newblock Simple matrix models for the flag, {G}rassmann, and {S}tiefel
  manifolds.
\newblock {\em arXiv:2407.13482}, 2024.

\bibitem{KN23}
K.~Z. Lin and N.~R. Zhang.
\newblock Quantifying common and distinct information in single-cell multimodal
  data with tilted canonical correlation analysis.
\newblock {\em Proc. Natl. Acad. Sci. USA}, 120(32):e2303647120, 2023.

\bibitem{Manivel01}
L.~Manivel.
\newblock {\em Symmetric functions, {S}chubert polynomials and degeneracy
  loci}, volume~6 of {\em SMF/AMS Texts and Monographs}.
\newblock American Mathematical Society, Providence, RI; Soci\'et\'e
  Math\'ematique de France, Paris, 2001.

\bibitem{MT24}
L.~Maxim and M.~Tib\u{a}r.
\newblock Euclidean distance degree and limit points in a {M}orsification.
\newblock {\em Adv. in Appl. Math.}, 152:Paper No. 102597, 20, 2024.

\bibitem{MRW1}
L.~G. Maxim, J.~I. Rodriguez, and B.~Wang.
\newblock Defect of {E}uclidean distance degree.
\newblock {\em Adv. in Appl. Math.}, 121:102101, 22, 2020.

\bibitem{MRW2}
L.~G. Maxim, J.~I. Rodriguez, and B.~Wang.
\newblock Euclidean distance degree of the multiview variety.
\newblock {\em SIAM J. Appl. Algebra Geom.}, 4(1):28--48, 2020.

\bibitem{MRW3}
L.~G. Maxim, J.~I. Rodriguez, and B.~Wang.
\newblock Euclidean distance degree of projective varieties.
\newblock {\em Int. Math. Res. Not. IMRN}, (20):15788--15802, 2021.

\bibitem{Milnor}
J.~Milnor.
\newblock {\em Morse theory}, volume No. 51 of {\em Annals of Mathematics
  Studies}.
\newblock Princeton University Press, Princeton, NJ, 1963.

\bibitem{OSV}
G.~Ottaviani, L.~Sodomaco, and E.~Ventura.
\newblock Asymptotics of degrees and {ED} degrees of {S}egre products.
\newblock {\em Adv. in Appl. Math.}, 130:Paper No. 102242, 36, 2021.

\bibitem{Schubert79}
H.~Schubert.
\newblock {\em Kalk\"ul der abz\"ahlenden {G}eometrie}.
\newblock Springer-Verlag, Berlin-New York, 1979.

\bibitem{optim5}
P.-J. Spaenlehauer.
\newblock On the complexity of computing critical points with {G}r\"obner
  bases.
\newblock {\em SIAM J. Optim.}, 24(3):1382--1401, 2014.

\bibitem{YL16}
K.~Ye and L.-H. Lim.
\newblock Schubert varieties and distances between subspaces of different
  dimensions.
\newblock {\em SIAM J. Matrix Anal. Appl.}, 37(3):1176--1197, 2016.

\end{thebibliography}

\end{document}